\documentclass{article}
\usepackage{amsmath,amssymb,amsthm}

\usepackage[colorlinks=true,linkcolor=blue,citecolor=blue,pdfpagelabels=false]{hyperref}

\theoremstyle{plain}
  \newtheorem{theorem}{Theorem}[section]

  \newtheorem{proposition}[theorem]{Proposition}
  
\theoremstyle{definition}
  
  \newtheorem{remark}[theorem]{Remark}

\newcommand{\email}[1]{{\textit{Email:} \texttt{#1}}}
\newcommand{\nocomma}{}
\newcommand{\noplus}{}

\newcommand{\op}[1]{\ensuremath{\operatorname{#1}}}

\newcommand{\assign}{:=}

\newcommand{\md}{\mathrm{d}}

\newcommand{\ZZ}{\mathbb{Z}}

\newcommand{\pFq}[5]{\ensuremath{{}_{#1}F_{#2} \left( \genfrac{}{}{0pt}{}{#3}{#4} \bigg| {#5} \right)}}

\begin{document}

\title{A solution of Sun's $\$520$ challenge concerning $\frac{520}{\pi}$}
\author{Mathew Rogers\thanks{\email{mathewrogers@gmail.com}}\\
Department of Mathematics and Statistics\\
Universit\'e de Montr\'eal\and Armin
Straub\thanks{\email{astraub@illinois.edu}}\\
Department of Mathematics\\
University of Illinois at Urbana-Champaign}\maketitle

\begin{abstract}
  We prove a Ramanujan-type formula for $520 / \pi$ conjectured by Sun. Our
  proof begins with a hypergeometric representation of the relevant double
  series, which relies on a recent generating function for Legendre
  polynomials by Wan and Zudilin. After showing that appropriate modular
  parameters can be introduced, we then apply standard techniques, going back
  to Ramanujan, for establishing series for $1 / \pi$.
\end{abstract}

\section{Introduction}

In \cite{sun-520}, Zhi-Wei Sun offered a $\$520$ prize for the first correct
proof of the following Ramanujan-type formula:
\begin{equation}
  \frac{520}{\pi} = \sum_{n = 0}^{\infty} \frac{1054 n + 233}{480^n}  \binom{2
  n}{n} \sum_{k = 0}^n \binom{n}{k}^2 \binom{2 k}{n} (- 1)^k 8^{2 k - n} .
  \label{eq:sunconj}
\end{equation}
The prize money attached to formula (\ref{eq:sunconj}) is in honor of the May
20th celebrations at Nanjing University. In this paper, we offer a complete
proof of (\ref{eq:sunconj}).

Our proof may be divided into three parts. First and crucially, in Section
\ref{sec:hyp} we connect series of the form (\ref{eq:sunconj}) to a product of
hypergeometric functions at algebraically related arguments. This first step
relies on the recent work {\cite{wz-legendre}} of Wan and Zudilin. Secondly, we
show in Section \ref{sec:modular} that the two algebraic arguments are values
of modular functions at simply related quadratic irrationalities. With such a
setup, there exist rather standard procedures to prove the corresponding series
for $1 / \pi$ (though we point out that series of the form
\eqref{eq:sunconjFGab} have only been considered since \cite{cwz-legendre}). To
varying degrees of detail, the approach employed here is, for instance, used in
{\cite{berndt-piseries2010}}, {\cite{cwz-legendre}} or {\cite{wz-legendre}}.
For the completeness of this note and for the convenience of the reader, we
chose to nevertheless include all necessary details in Section \ref{sec:exec}.
For a historical account and many more references on series for $1 / \pi$,
originating with {\cite{ramanujan1}}, we refer to {\cite{berndt-piseries2010}}.

Sun's list of identities {\cite{sun-520}} contains many formulas along the
lines of \eqref{eq:sunconj}, which were discovered by computational searches guided by the
observation that most series for $1/\pi$ like \eqref{eq:sunconj} appear to
occur in tandem with congruences for their finite sum analogs. This is
illustrated in \cite{sun-135c} where \eqref{eq:sunconj} appears as part of
Conjecture 7.11.  The final Section \ref{sec:other} briefly discusses the applicability of our
approach to proving all the identities in {\cite{sun-520,sun-135c}} of the same form as
(\ref{eq:sunconj}).

\section{Step I: A hypergeometric representation}\label{sec:hyp}

Since virtually every interesting formula for $1 / \pi$ comes from studying
hypergeometric functions, we expect that the same should be true for
(\ref{eq:sunconj}). Let us define
\begin{equation}
  A (x, y) = \sum_{n = 0}^{\infty} x^n  \binom{2 n}{n} \sum_{k = 0}^n
  \binom{n}{k}^2 \binom{2 k}{n} (- 1)^k y^{2 k - n} . \label{eq:Fdef}
\end{equation}
Clearly then Sun's conjecture (\ref{eq:sunconj}) is equivalent to
\begin{equation}
  233 A \left( \frac{1}{480}, 8 \right) + 1054 \left( \theta_x A \right)
  \left( \frac{1}{480}, 8 \right) = \frac{520}{\pi} \label{eq:sunconjA}
\end{equation}
where $\theta_x = x \frac{\md}{\md x}$ is the Euler differential
operator. Our first step is to reexpress $A (x, y)$ as a series of
Legendre polynomials. 
Recent work of Wan and Zudilin \cite{wz-legendre} then makes it possible to
write $A(x,y)$ as a product of hypergeometric functions.

\begin{proposition}\label{prop:AxyP}
  Let $x$ and $y$ be such that (\ref{eq:Fdef}) converges absolutely. Then
  \begin{equation}
    A (x,y) = \sum_{k = 0}^{\infty} \left( - x y \right)^k 
    \binom{2 k}{k}^2 P_{2 k} \left( \sqrt{1 + \frac{4 x}{y}} \right) .
    \label{eq:AxyP}
  \end{equation}
\end{proposition}

\begin{proof}
  If (\ref{eq:Fdef}) converges absolutely, then we can interchange the order
  of summation to obtain
  \begin{align*}
    A (x,y) & = \sum_{n = 0}^{\infty} x^n  \binom{2 n}{n}
    \sum_{k = 0}^n \binom{n}{k}^2 \binom{2 k}{n} (- 1)^k y^{2 k - n}\\
    & = \sum_{k = 0}^{\infty} \sum_{n = k}^{\infty} x^n  \frac{\left( 2 n
    \right) !}{k!^2 \left( n - k \right) !^2}  \frac{\left( 2 k \right)
    !}{\left( 2 k - n \right) !n!} (- 1)^k y^{2 k - n}\\
    & = \sum_{k = 0}^{\infty} \sum_{n = 0}^{\infty} x^{n + k}  \frac{\left(
    2 n + 2 k \right) !}{k!^2 n!^2}  \frac{\left( 2 k \right) !}{\left( k - n
    \right) ! \left( n + k \right) !} (- 1)^k y^{k - n}\\
    & = \sum_{k = 0}^{\infty} \left( - x y \right)^k  \binom{2 k}{k}^2
    \pFq21{1/2+k,-k}{1}{-\frac{4x}{y}} .
  \end{align*}
  By the quadratic transformation {\cite[(3.1.3)]{aar}} for the hypergeometric
  function, we have
  \begin{align*}
    \pFq21{1/2+k,-k}{1}{-\frac{4x}{y}}
    &= \pFq21{1+2k,-2k}{1}{\frac{1-\sqrt{1+\frac{4x}{y}}}{2}} \\
    & = P_{2 k} \left( \sqrt{1 + \frac{4 x}{y}} \right)
  \end{align*}
  where $P_k (z)$ is the $k$th Legendre polynomial.
  This is (\ref{eq:AxyP}).
\end{proof}

The proof of Proposition \ref{prop:AxyP} actually provides a convenient
criterion for absolute convergence of the original double series
\eqref{eq:Fdef}.

\begin{proposition}\label{prop:conv}
  The double series \eqref{eq:Fdef} converges absolutely for $x, y$ if
  \begin{equation}\label{eq:convratio}
    \left| 16 x y \left( \sqrt{\left| \frac{4 x}{y} \right|} + \sqrt{1 +
    \left| \frac{4 x}{y} \right|} \right)^2 \right| < 1.
  \end{equation}
\end{proposition}

\begin{proof}
  Absolute convergence of $A(x,y)$, as in \eqref{eq:Fdef}, is equivalent to
  convergence of
  \begin{equation*}
    A (i |x|, i |y|) = \sum_{n = 0}^{\infty} |x|^n \binom{2 n}{n} \sum_{k = 0}^n
    \binom{n}{k}^2 \binom{2 k}{n} |y|^{2 k - n} .
  \end{equation*}
  The fact that all summands are positive again justifies interchanging the order
  of summation as in the proof of Proposition \ref{prop:AxyP} to obtain
  \begin{equation*}
    A (i |x|, i |y|) =
    \sum_{k = 0}^{\infty} |x y|^k 
    \binom{2 k}{k}^2 P_{2 k} \left( \sqrt{1 + \frac{4 |x|}{|y|}} \right) .
  \end{equation*}
  The classical asymptotic formula for the Legendre
  polynomials, e.g. {\cite[(8.21.1)]{szego3}}, states that
  \begin{equation}\label{eq:Pasy}
    P_n (z) \cong \frac{\left( z + \left( z^2 - 1 \right)^{1 /
    2} \right)^{n + 1 / 2}}{\left( 2 \pi n \right)^{1 / 2} \left( z^2 - 1
    \right)^{1 / 4}},
  \end{equation}
  as $n \rightarrow \infty$ and is valid for $z$ outside the interval $\left[ -
  1, 1 \right]$. Here, in the cases of interest, $z>1$ and combining
  \eqref{eq:Pasy} with the well-known asymptotics for the central binomial
  coefficients shows \eqref{eq:convratio}.
\end{proof}

In particular, the series (\ref{eq:Fdef}) converges absolutely for $(x,y) =
\left( \frac{1}{480}, 8 \right)$.
Moreover, we find that in this case the convergence of \eqref{eq:AxyP} is
geometric with ratio $-\frac{64}{225} = -\left( \frac{8}{15} \right)^2$.

\begin{theorem}
  \label{thm:Ahyp}
  Let $x$ and $y$ be such that \eqref{eq:AxyP} holds.
  Suppose, further, there are $X$ and $Y$ in a certain neighborhood of $1$
  such that
  \begin{equation}
    - x y = \left( \frac{X - Y}{4 \left( 1 \noplus + X Y \right)} \right)^2,
    \hspace{2em} 1 + \frac{4 x}{y} = \left[ \frac{\left( X + Y \right) \left(
    1 - X Y \right)}{\left( X - Y \right) \left( 1 + X Y \right)} \right]^2 .
    \label{eq:relXYxy}
  \end{equation}
  Then
  \begin{equation}\label{eq:Ahyp}
    A (x, y) = \frac{1 + X Y}{2} \pFq21{1/2,1/2}{1}{1-X^2} \pFq21{1/2,1/2}{1}{1-Y^2}.
  \end{equation}
\end{theorem}

\begin{proof}
  A series of the form \eqref{eq:AxyP} was recently summed by Wan
  and Zudilin {\cite{wz-legendre}}. They proved that when $X$ and $Y$ lie in a
  certain neighborhood of $1$, then
  \begin{align*}
    &  \sum_{k = 0}^{\infty} \left( \frac{X - Y}{4 \left( 1 \noplus + X Y
    \right)} \right)^{2 k}  \binom{2 k}{k}^2 P_{2 k} \left( \frac{\left( X + Y
    \right) \left( 1 - X Y \right)}{\left( X - Y \right) \left( 1 + X Y
    \right)} \right)\\
    & = \frac{1 + X Y}{2} \pFq21{1/2,1/2}{1}{1-X^2} \pFq21{1/2,1/2}{1}{1-Y^2}.
  \end{align*}
  If $X$ and $Y$ are chosen as in (\ref{eq:relXYxy}), this is equal to $A
  (x,y)$ as claimed.
\end{proof}

\begin{remark}
  \label{rk:XYsym}The solutions to the equations (\ref{eq:relXYxy}) have a
  remarkable number of symmetries. Namely, if $(X,Y)$ is a
  solution, then so are $\left( - X, - Y \right)$, $\left( X^{- 1}, Y^{- 1}
  \right)$, and $\left( Y, X \right)$. Assuming that $x$ and $y$ are real, we
  also have $\left( \bar{X}, \bar{Y} \right)$ as another solution. Less
  obviously, if $(X,Y)$ is a solution, then so is
  \begin{equation}
    \left( \frac{1 - X}{1 + X}, \frac{1 - Y}{1 + Y} \right) .
    \label{eq:XYtwist}
  \end{equation}
\end{remark}

\section{Notation}\label{sec:notation}

Before proceeding with the proof of Sun's identity, we introduce some very
standard objects from elliptic modular function theory. Unfortunately,
notation is not as standard.

Throughout, $q = e^{2 \pi i \tau}$ and $\tau$ is assumed to lie in the upper
half-plane. We denote $D = \frac{1}{2 \pi i}  \frac{\md}{\md \tau} = q
\frac{\md}{\md q}$. The Dedekind eta function is
\[ \eta (\tau) = q^{1 / 24} \prod_{n \geqslant 1} (1 - q^n) \]
and its logarithmic derivative determines the Eisenstein series
\[ E_2 (\tau) = 24 \frac{D \eta}{\eta} = 1 + 24 \sum_{n \geqslant
   1} \frac{n q^n}{1 - q^n} . \]
The standard Jacobi theta functions are
\[ \theta_2 (\tau) = \sum_{n = - \infty}^{\infty} q^{\left( n + 1
   / 2 \right)^2 / 2} \nocomma, \hspace{1em} \theta_3 (\tau) =
   \sum_{n = - \infty}^{\infty} q^{n^2 / 2}, \hspace{1em} \theta_4 (\tau) = \sum_{n = - \infty}^{\infty} \left( - 1 \right)^n q^{n^2 / 2} .
\]
They can be used to express the elliptic modulus $k (\tau)$ as
well as the complementary modulus $k' (\tau)$ as
\[ k (\tau) = \left( \frac{\theta_2 (\tau)}{\theta_3
   (\tau)} \right)^2, \hspace{2em} k' (\tau) =
   \left( \frac{\theta_4 (\tau)}{\theta_3 (\tau)}
   \right)^2 . \]
We note the classical $k^2 + \left( k' \right)^2 = 1$ as well as the
transformations
\begin{equation}
  k (\tau+1) = i \frac{k (\tau)}{k' (\tau)}, \hspace{1em}
  k' (\tau+1) = \frac{1}{k' (\tau)}, \hspace{1em}
  k \left( - \frac{1}{\tau} \right) = k' (\tau) . \label{eq:ktr}
\end{equation}
Finally, the complete elliptic integral of the first kind is
\begin{equation*}
  K(k) = \frac\pi2 \pFq21{1/2,1/2}{1}{k^2}
\end{equation*}
With this setup we have the classical identity
\begin{equation}\label{eq:Kktheta}
  \frac{2}{\pi} K \left( k (\tau) \right)
  = \pFq21{1/2,1/2}{1}{k^2(\tau)}
  = \theta_3 ( \tau )^2,
\end{equation}
valid in a neighborhood of $i \infty$, which in light of Theorem
\ref{thm:Ahyp} underlies all that follows.

\section{Step II: Introducing modular parameters}\label{sec:modular}

\begin{proposition}
  \label{prop:tau}If $x = \frac{1}{480}$ and $y = 8$, then
  \[ X = k' (\tau_0), \hspace{2em} Y = k' (5\tau_0), \]
  with $\tau_0 = \frac{1}{2} + \frac{3}{10} \sqrt{- 5}$, satisfy the two
  relations (\ref{eq:relXYxy}).
\end{proposition}

\begin{proof}
  Note that $X$ and $Y$ are both values of modular functions (with respect to
  some congruence subgroup) at a quadratic irrational point in the upper
  half-plane. By standard principles, $X$ and $Y$ are therefore effectively
  computable algebraic numbers {\cite{coxCM}}, {\cite{borwein-piagm}}. In the
  case at hand, we find that $X$ and $Y$ both have the palindromic minimal
  polynomial $z^8 p \left( z^2 + 1 / z^2 \right)$ where
  \[ p (z) = z^4 + 88796296 z^3 + 237562136 z^2 - 595063264 z -
     470492144. \]
  Together with the approximate values $X \approx 0.57884718 - 0.81543604 i$
  and $Y \approx 0.99999998 - 0.00021224 i$ this determines both as algebraic
  numbers. It is then computational routine to verify that the two relations
  (\ref{eq:relXYxy}) are indeed satisfied.
\end{proof}

Nevertheless, in order to obtain explicit radical expressions for $X$ and $Y$,
we sketch in Remark \ref{rk:XYalg} how one may alternatively use modular
equations and known singular values to evaluate $k' (\tau_0)$ and
$k' (5\tau_0)$ without assistance by a computer algebra system.
Ultimately, this allows us to give rather appealing radical expressions in
Remark \ref{rk:XYrad}.

\begin{remark}
  \label{rk:XYalg}We denote the singular values in the usual way as $k_n = k
  \left( i \sqrt{n} \right)$ and $k_n' = k' \left( i \sqrt{n} \right) = k
  \left( i / \sqrt{n} \right)$. Then Ramanujan's class invariant $G_n$
  satisfies (or may be introduced as)
  \[ G_n^{12} = \frac{1}{2 k_n k_n'} . \]
  The reason for introducing this additional notation is that the values $G_n$
  and $k_n$ (via $\alpha_n \assign k_n^2$) are extensively tabulated
  {\cite[Chapter 34]{berndtV}}. We observe that
  \begin{equation}
    k' (\tau) = \frac{1 - k \left( 2 \tau \right)}{1 + k \left( 2
    \tau \right)} \label{eq:mod2k}
  \end{equation}
  is a particular form of the modular equation of degree $2$. Hence to find
  $k' (5\tau_0)$ we need only evaluate $k (10\tau_0)$. To this end, using (\ref{eq:ktr}),
  \[ k (10\tau_0) = k \left( 1 + 3 i \sqrt{5} \right) = i
     \frac{k_{45}}{k'_{45}} = 2 i k_{45}^2 G_{45}^{12} . \]
  Equipped with the values $k_{45}^2 = \alpha_{45}$ and $G_{45}$ found in
  {\cite[Chapter 34]{berndtV}}, we then obtain
  \begin{equation}
    k (10\tau_0) = i \left( \sqrt{\frac{7 + 3 \sqrt{5}}{4}} -
    \sqrt{\frac{3 + 3 \sqrt{5}}{4}} \right)^4  \left( \sqrt{\frac{3 +
    \sqrt{5}}{2}} - \sqrt{\frac{1 + \sqrt{5}}{2}} \right)^4 \label{eq:k10tau}
  \end{equation}
  and hence have evaluated
  \begin{equation}
    Y = k' (5\tau_0) = \frac{1 - k (10\tau_0)}{1 +
    k (10\tau_0)} . \label{eq:Yalg}
  \end{equation}
  Similar considerations apply for $X = k' (\tau_0)$ (which in
  principle can also be determined from (\ref{eq:Yalg}) and the modular
  equation (\ref{eq:mod5}) of degree $5$).
\end{remark}

\begin{remark}
  \label{rk:XYrad}It can be seen from (\ref{eq:mod2k}) and (\ref{eq:XYtwist})
  that the numbers $k' (5\tau_0)$ and $k (10\tau_0)$ are conjugate as algebraic numbers. In light of (\ref{eq:k10tau})
  the evaluation (\ref{eq:Yalg}) may therefore be further simplified. Indeed,
  we have the radical expressions
  \begin{align*}
    X & = i \left( \sqrt{\frac{7 - 3 \sqrt{5}}{4}} - \sqrt{\frac{3 - 3
    \sqrt{5}}{4}} \right)^4  \left( \sqrt{\frac{3 - \sqrt{5}}{2}} -
    \sqrt{\frac{1 - \sqrt{5}}{2}} \right)^4,\\
    Y & = i \left( \sqrt{\frac{7 - 3 \sqrt{5}}{4}} - \sqrt{\frac{3 - 3
    \sqrt{5}}{4}} \right)^4  \left( \sqrt{\frac{3 - \sqrt{5}}{2}} +
    \sqrt{\frac{1 - \sqrt{5}}{2}} \right)^4 .
  \end{align*}
  Each factor in these expressions is an algebraic unit (and hence so are $X$
  and $Y$). In particular, the product $X Y$ and the quotient $X / Y$ take the
  particularly simple forms
  \[ X Y = - \left( \sqrt{\frac{7 - 3 \sqrt{5}}{4}} - \sqrt{\frac{3 - 3
     \sqrt{5}}{4}} \right)^8, \hspace{1em} \frac{X}{Y} = \left( \sqrt{\frac{3
     - \sqrt{5}}{2}} - \sqrt{\frac{1 - \sqrt{5}}{2}} \right)^8 . \]
  We also note that using the stated radical expressions it is trivial to
  verify that $X$ and $Y$ satisfy the two relations (\ref{eq:relXYxy}).
\end{remark}

\begin{remark}
  Let us indicate how one may find the value of $\tau_0$ used in Proposition
  \ref{prop:tau}. Let $x = \frac{1}{480}$ and $y = 8$ and assume that the
  numbers $X, Y$ satisfy the two relations (\ref{eq:relXYxy}) which, upon
  clearing denominators, are polynomial. By computing resultants, we conclude
  that $X$ and $Y$ both have minimal polynomial $q (z)$ as stated
  in the proof of Proposition \ref{prop:tau}. For each root $X$ of $q (z)$, we can now find $\tau_0$ such that $X = k' (\tau_0)$
  using the classical inversion relation
  \[ \tau = i \frac{K \left( k' \right)}{K \left( k \right)} \]
  with the values $k' = X$ and $k = \sqrt{1 - X^2}$. With $\tau_0$ such
  determined, $k' (\tau_0)^2 = X^2$. This allows $\tau_0$ to be
  numerically computed from $X$ and it only remains to recognize $\tau_0$, as
  is possible here, as a quadratic irrationality.
  
  Different choices of roots $X$ of $q (z)$ lead to different
  values for $\tau_0$ but these are related by the symmetries observed in
  Remark \ref{rk:XYsym}. Essentially, this leaves two choices: $\tau_0$ as
  presented in Proposition \ref{prop:tau} and, related by (\ref{eq:mod2k}),
  $\tau_1 = - \frac{1}{10 \tau_0}$ which also satisfies the statement of
  Proposition \ref{prop:tau}. The second choice, however, results in values
  for $X, Y$ which are not sufficiently close to $1$ in order to use
  (\ref{eq:Ahyp}). This is further explained in Remark \ref{rk:tau1fail}.
\end{remark}

Following the notation of {\cite{cwz-legendre}} let us denote
\begin{equation*}
  F(\alpha) = \pFq21{1/2,1/2}{1}{\alpha}, \hspace{1em}
  G(\alpha) = \alpha \frac{\md}{\md\alpha} F(\alpha) = \frac\alpha4 \pFq21{3/2,3/2}{2}{\alpha}.
\end{equation*}
Throughout, we assume the standard branch cut for $F$ and $G$, with a cut
along the real axis from $1$ to $\infty$.

\begin{proposition}
  \label{prop:tauA}Let $x = \frac{1}{480}$ and $y = 8$. Let $X = k' (\tau_0)$, $Y = k' (5\tau_0)$ for $\tau_0 = \frac{1}{2}
  + \frac{3}{10} \sqrt{- 5}$ as in Proposition \ref{prop:tau}. Then
  (\ref{eq:Ahyp}) holds in a neighborhood of these values. In particular,
  \begin{equation}
    A (x, y) = \frac{1 + X Y}{2} F (1-X^2) F (1-Y^2) . \label{eq:AFXY}
  \end{equation}
\end{proposition}

\begin{proof}
  We note that we cannot rely directly on Theorem \ref{thm:Ahyp} because $X$
  does not lie in a sufficient neighborhood of $1$. In fact, $\left| 1 - X^2
  \right| \approx 1.63087$ from the approximations for $X$ and $Y$ that were
  noted in the proof of Proposition \ref{prop:tau}.
  
  Throughout, let $t \in \left[ 0, 1 \right)$. We consider the convex
  combinations
  \[ X_t = \left( 1 - t \right) X + t, \hspace{1em} Y_t = \left( 1 - t \right)
     Y + t, \]
  along with the corresponding values $x_t, y_t$, determined by
  (\ref{eq:relXYxy}). Our strategy is as follows: for $t_0 < 1$ large enough
  the values $X_{t_0}, Y_{t_0}$ are sufficiently close to $1$ so that Theorem
  \ref{thm:Ahyp} applies, showing that, for $t = t_0$,
  \begin{equation}
    A (x_t, y_t) = \frac{1 + X_t Y_t}{2} F \left( 1 - X_t^2 \right) F \left( 1
    - Y_t^2 \right) . \label{eq:AFXYt}
  \end{equation}
  We wish to conclude that (\ref{eq:AFXY}) holds by analytically continuing
  (\ref{eq:AFXYt}) from $t = t_0$ to $t = 0$. By construction, $X_t, Y_t$ are
  not purely imaginary and hence the arguments $1 - X_t^2$ and $1 - Y_t^2$ of
  $F$ stay off the branch cut. We are thus done if we can show that, for all
  $t \in \left[ 0, t_0 \right]$, the values $x_t, y_t$ are such that the
  series (\ref{eq:AxyP}) for $A \left( x_t, y_t \right)$ converges absolutely.
  That this is crucial is illustrated by Remark \ref{rk:tau1fail}.
  
  To begin, we note that
  \begin{align}
    1 - X_t Y_t & = \left( 1 - t \right) \left( 1 - X Y + t \left( 1 - X
    \right) \left( 1 - Y \right) \right),  \label{eq:est1}\\
    \left( 1 + X_t Y_t \right) - \left( X_t + Y_t \right)
    & = \left( 1 - t \right)^2 \left( 1 - X \right) \left( 1 - Y \right) .  \label{eq:est2}
  \end{align}
  Using (\ref{eq:est1}) we then estimate
  \[ \left| 1 - X_t Y_t \right| \leqslant \left| 1 - X Y \right| + \left| 1 -
     X \right|  \left| 1 - Y \right| \leqslant 0.92 \]
  where the final inequality simply holds because with our values of $X$ and
  $Y$ (recorded in the proof of Proposition \ref{prop:tau}) the middle
  quantity is, to three digits, equal to $0.918$. Consequently,
  \[ \left| 1 + X_t Y_t \right| \geqslant 2 - \left| 1 - X_t Y_t \right|
     \geqslant 1.08, \]
  from which we conclude
  \begin{align}
    \left| x_t y_t \right|^{1 / 2} & = \left| \frac{X_t - Y_t}{4 \left( 1
    \noplus + X_t Y_t \right)} \right| \leqslant \frac{\left| X_t - Y_t
    \right|}{4 \cdot 1.08} \leqslant \frac{\left| X - Y \right|}{4 \cdot 1.08}
    \leqslant 0.22. \label{eq:estA} 
  \end{align}
  Once more, the final inequality depends on our specific values of $X$ and
  $Y$. Using both (\ref{eq:est1}) and (\ref{eq:est2}),
  \begin{align}
    1 + \frac{4 x_t}{y_t} & = \left[ \frac{1 - X_t Y_t}{X_t - Y_t} \right]^2
    \left[ \frac{X_t + Y_t}{1 + X_t Y_t} \right]^2 \nonumber\\
    & = \left[ \frac{1 - X Y}{X - Y} + \frac{t \left( 1 - X \right) \left(
    1 - Y \right)}{X - Y} \right]^2 \left[ 1 - \frac{\left( 1 - t \right)^2
    \left( 1 - X \right) \left( 1 - Y \right)}{1 + X_t Y_t} \right]^2
    \nonumber\\
    & = : \left[ A + \varepsilon_1 \right]^2 \left[ 1 + \varepsilon_2
    \right]^2 .  \label{eq:estq}
  \end{align}
  The smallness of the second summands $\varepsilon_1$ and $\varepsilon_2$ is
  witnessed by
  \[ \left| \varepsilon_1 \right| = \left| \frac{t \left( 1 - X \right) \left(
     1 - Y \right)}{X - Y} \right| \leqslant \left| \frac{\left( 1 - X \right)
     \left( 1 - Y \right)}{X - Y} \right| \leqslant 0.00022 \]
  and
  \[ \left| \varepsilon_2 \right| = \left| \frac{\left( 1 - t \right)^2 \left(
     1 - X \right) \left( 1 - Y \right)}{1 + X_t Y_t} \right| \leqslant
     \frac{\left| \left( 1 - X \right) \left( 1 - Y \right) \right|}{1.08}
     \leqslant 0.00019. \]
  Combining these with $\left| A \right| \leqslant 1.00042$, expanding
  (\ref{eq:estq}), subtracting $1$, and using the triangle inequality, we find
  \begin{equation}
    \left| \frac{4 x_t}{y_t} \right| \leqslant \left| \left( \frac{1 - X Y}{X
    - Y} \right)^2 - 1 \right| + 0.00083 \leqslant 0.0017. \label{eq:estB}
  \end{equation}
  The estimates (\ref{eq:estA}) and (\ref{eq:estB}) together with
  (\ref{eq:convratio}) now show that the series (\ref{eq:AxyP}) converges
  absolutely for all pairs $x_t, y_t$ (the estimates show that the convergence
  is geometric with ratio at most $0.85$).
\end{proof}

\begin{remark}
  \label{rk:tau1fail}It follows from (\ref{eq:XYtwist}) and (\ref{eq:mod2k})
  that $\tau_1 = - \frac{1}{10 \tau_0}$ also satisfies the statement of
  Proposition \ref{prop:tau}. However, Proposition \ref{prop:tauA} fails to
  hold with
  \begin{align*}
    X & = k' (\tau_1) = k (10\tau_0) \approx
    0.000106121305 i,\\
    Y & = k' (5\tau_1) = k \left( 2 \tau_0 \right) \approx
    0.51647560 i.
  \end{align*}
  The fact that the corresponding $1 - X^2$ and $1 - Y^2$ are on the branch
  cut of $F$ is not the reason for the failing of (\ref{eq:AFXY}), as can be
  seen from the fact that, for $t > 0$, the $1 - X_t^2, 1 - Y_t^2$ avoid the
  branch cut so that the monodromy of $F$ is not invoked by ever crossing the
  cut. However, the values $x_t, y_t$ in the process are such that the domain
  of convergence of $A$, as defined in (\ref{eq:AxyP}), is first left and then
  reentered again. Thus the failure of (\ref{eq:AFXY}) is due to the monodromy
  of $A$. This illustrates the need for an argument such as the one given in
  the proof of Proposition \ref{prop:tauA}. Clearly, it would be nice to
  replace it with a more conceptual and concise argument.
\end{remark}

\section{Step III: The execution}\label{sec:exec}

With Proposition \ref{prop:tauA} established, Sun's conjecture
(\ref{eq:sunconjA}) may now be confirmed using precisely the techniques of,
for instance, {\cite{berndt-piseries2010}}, {\cite{cwz-legendre}} and
{\cite{wz-legendre}}. It is characteristic for this technique to produce
unwieldy algebraic numbers along the way. Only in a few cases have we
therefore chosen to include these numbers explicitely by giving their minimal
polynomial along with a sufficient numerical approximation. Of course, most
computer algebra systems are very comfortable with symbolically manipulating
algebraic numbers (here, the algebraic degree never exceeds $16$) and it is in
some way rather fitting that computer algebra is at work when establishing
conjectures that were found by computational search.

\begin{proposition}
  \label{prop:Ax}Suppose that (\ref{eq:Ahyp}) holds. Then we have
  \[ \frac{\md A}{\md x} = r (X,Y) G (1-X^2) F (1-Y^2) + s (X,Y) F (1-X^2) G (1-Y^2) \]
  where $r$ and $s$ are explicit rational functions.
\end{proposition}

\begin{proof}
  We note that
  \[ \frac{\md}{\md X} F (1-X^2) = \frac{- 2 X}{1 - X^2}
     G (1-X^2) \]
  and hence, differentiating (\ref{eq:Ahyp}),
  \begin{align*}
    2 \frac{\md A}{\md x} & = \frac{\md A}{\md X}  \frac{\md
    X}{\md x} + \frac{\md A}{\md Y}  \frac{\md Y}{\md x}\\
    & = \left[ Y F (1-X^2) - \frac{2 X \left( 1 + X Y
    \right)}{1 - X^2} G (1-X^2) \right] F (1-Y^2)  \frac{\md X}{\md x}\\
    &\quad  + \left[ X F (1-Y^2) - \frac{2 Y \left( 1 + X Y
    \right)}{1 - Y^2} G (1-Y^2) \right] F (1-X^2)  \frac{\md Y}{\md x} .
  \end{align*}
  Further, $\frac{\md X}{\md x}$ and $\frac{\md Y}{\md x}$ are
  determined by differentiating the two relations (\ref{eq:relXYxy}) which
  gives
  \begin{align*}
    - y & = \frac{X - Y}{8 \left( 1 + X Y \right)^3} \left[ \left( 1 + Y^2
    \right) \frac{\md X}{\md x} - \left( 1 + X^2 \right) \frac{\md
    Y}{\md x} \right],\\
    \frac{4}{y} & = \frac{4 \left( X + Y \right) \left( 1 - X Y
    \right)}{\left[ \left( X - Y \right) \left( 1 + X Y \right) \right]^3}
    \bigg[ X \left( 1 + Y^2 \right) (1-X^2) \frac{\md
    Y}{\md x} \\
    &\quad\quad- Y \left( 1 + X^2 \right) (1-Y^2)
    \frac{\md X}{\md x} \bigg] .
  \end{align*}
  Combining these results in the claimed combination.
\end{proof}

For our purposes it will become convenient to employ the algebraic numbers
\begin{equation}
  \alpha \assign 1 - X^2 = k^2 (\tau_0), \hspace{2em} \beta
  \assign 1 - Y^2 = k^2 (5\tau_0) \label{eq:ab}
\end{equation}
with $\tau_0$ as in Proposition \ref{prop:tau}. For the convenience of the
reader, we record their numerical values as
\[ \alpha \approx 1.329871878 + 0.944025712 i, \hspace{1em} \beta \approx
   9.00938 \times 10^{- 8} + 0.0004244852051 i. \]
We can then express the left-hand side of (\ref{eq:sunconjA}) in terms of the
functions $F$ and $G$ at the values $\alpha$ and $\beta$.

\begin{proposition}
  With $\alpha, \beta$ as in (\ref{eq:ab}), Conjecture (\ref{eq:sunconj}) is
  equivalent to
  \begin{equation}\label{eq:sunconjFGab}
    r_1 F (\alpha) F (\beta) + r_2 G \left( \alpha
    \right) F (\beta) + r_3 F (\alpha) G \left(
    \beta \right) = \frac{520}{\pi}
  \end{equation}
  where $r_1, r_2, r_3$ are explicitly computable algebraic numbers.
\end{proposition}

\begin{proof}
  Combining Proposition \ref{prop:tauA} with Proposition \ref{prop:Ax}, we
  find that the left-hand side of (\ref{eq:sunconjA}) equals
  \[ 233 \frac{1 + X Y}{2} F (\alpha) F (\beta) +
     \frac{1054}{480} r (X,Y) G (\alpha) F \left(
     \beta \right) + s (X,Y) F (\alpha) G \left(
     \beta \right) \]
  with $r, s$ as in Proposition \ref{prop:Ax}. Since $X$ and $Y$ are
  algebraic, the claim follows.
\end{proof}

As in {\cite[Proposition 3]{cwz-legendre}} we next express the values $F
(\beta), G (\beta)$ by the corresponding values $F
(\alpha), G (\alpha)$.

\begin{proposition}
  \label{prop:FGb}Let $\alpha, \beta$ be as in (\ref{eq:ab}). Then there are
  explicitly computable algebraic numbers $t, t_1, t_2$ such that
  \begin{itemize}
    \item $F (\beta) = t F (\alpha)$,
    
    \item $G (\beta) = t_1 F (\alpha) + t_2 G \left(
    \alpha \right)$.
  \end{itemize}
\end{proposition}

\begin{proof}
  In the language of Ramanujan, $\beta = k^2 (5\tau)$ has
  degree 5 over $\alpha = k^2 (\tau)$, whence the numbers
  satisfy the modular equation of degree $5$ {\cite[Entry 19.13]{berndtIII}}
  \begin{equation}
    (\alpha \beta)^{1 / 2} + \left\{ (1 - \alpha) (1 - \beta) \right\}^{1 / 2}
    + 2 \left\{ 16 \alpha \beta (1 - \alpha) (1 - \beta) \right\}^{1 / 6} = 1.
    \label{eq:mod5}
  \end{equation}
  Differentiating with respect to $\alpha$ yields
  \begin{align*}
    &\frac{\beta^{1 / 2}}{\alpha^{1 / 2}} - \frac{\left( 1 - \beta \right)^{1 /
    2}}{\left( 1 - \alpha \right)^{1 / 2}} + 2 \frac{1 - 2 \alpha}{3} 
    \frac{\left[ 16 \beta \left( 1 - \beta \right) \right]^{1 / 6}}{\left[
    \alpha \left( 1 - \alpha \right) \right]^{5 / 6}} \\
    &\quad+ \left[ \frac{\alpha^{1 / 2}}{\beta^{1 / 2}} - \frac{\left( 1 - \alpha
    \right)^{1 / 2}}{\left( 1 - \beta \right)^{1 / 2}} + 2 \frac{1 - 2
    \beta}{3}  \frac{\left[ 16 \alpha \left( 1 - \alpha \right) \right]^{1 /
    6}}{\left[ \beta \left( 1 - \beta \right) \right]^{5 / 6}} \right]
    \frac{\md \beta}{\md \alpha} = 0,
  \end{align*}
  which, upon taking the appropriate roots, provides us with $\frac{\md
  \beta}{\md \alpha}$ as an algebraic function of $\alpha$.
  
  Moreover, the multiplier $\frac{F (\alpha)}{F \left( \beta
  \right)}$ may be expressed as {\cite[Entry 19.13]{berndtIII}}
  \begin{equation}
    \frac{F (\alpha)}{F (\beta)} = \left(
    \frac{\beta}{\alpha} \right)^{1 / 4} + \left( \frac{1 - \beta}{1 - \alpha}
    \right)^{1 / 4} - \left( \frac{\beta \left( 1 - \beta \right)}{\alpha
    \left( 1 - \alpha \right)} \right)^{1 / 4}, \label{eq:mult}
  \end{equation}
  which verifies the first part of the claim.
  
  Further, taking the logarithmic derivative of (\ref{eq:mult}) with respect
  to $\alpha$, we find
  \begin{align*}
    \frac{G (\alpha)}{\alpha F (\alpha)} - \frac{G
    (\beta)}{\beta F (\beta)}  \frac{\md
    \beta}{\md \alpha} & = \frac{\md}{\md \alpha} \left[ \left(
    \frac{\beta}{\alpha} \right)^{1 / 4} + \left( \frac{1 - \beta}{1 - \alpha}
    \right)^{1 / 4} - \left( \frac{\beta \left( 1 - \beta \right)}{\alpha
    \left( 1 - \alpha \right)} \right)^{1 / 4} \right] .
  \end{align*}
  Noting that $\frac{\md \beta}{\md \alpha}$ as well as the right-hand
  side are algebraic functions of $\alpha$, the second claim follows from the first.
\end{proof}

\begin{remark}\label{rk:FGb}
  We briefly note that Proposition \ref{prop:FGb} holds in much greater generality.
  For instance, let $\alpha=k^2(\tau)$ and $\beta=k^2(M\cdot\tau)$ where $M\in\op{GL}_2(\ZZ)$ is an
  integral matrix acting by a linear fractional transformation.
  Then $\alpha,\beta$ are both modular functions (with respect to some congruence subgroup) and
  hence algebraically related (as made explicit by \eqref{eq:mod5} when $M$ acts by multiplication with $5$).
  Moreover, because of \eqref{eq:Kktheta}, $F(\alpha)$ and $F(\beta)$ are both modular forms of
  weight $1$. Hence the quotient ${F(\alpha)}/{F(\beta)}$ is a modular function and again an algebraic
  function of $\alpha$ (as illustrated by \eqref{eq:mult}).
  Thus if $\tau$ is a quadratic irrationality, Proposition \ref{prop:FGb} holds for $\alpha,\beta$ as well.
\end{remark}

At this stage, we have expressed Sun's conjecture (\ref{eq:sunconj}) in the
form
\begin{equation}
  r_1 F (\alpha)^2 + r_2 F (\alpha) G ( \alpha ) = \frac{520}{\pi} \label{eq:sunconjFGa}
\end{equation}
where $r_1, r_2$ are explicit algebraic numbers.

\begin{proposition}
  \label{prop:FGE2}With $\alpha = k^2 (\tau)$, we have
  \[ F (\alpha) G (\alpha) = \frac{E_2 ( \tau ) + ( 2 \alpha - 1 ) F (\alpha)^2}{6 (1 - \alpha)} . \]
\end{proposition}

\begin{proof}
  This is {\cite[Entry 17.9(iv)]{berndtIII}} in mildly different notation.
\end{proof}

Applying Proposition \ref{prop:FGE2} with $\alpha$ as in (\ref{eq:ab}),
equation (\ref{eq:sunconjFGa}) becomes
\begin{equation}
  s_1 F (\alpha)^2 + 52 \sqrt{5} E_2 ( \tau_0 ) =
  \frac{520}{\pi} \label{eq:sunFE2}
\end{equation}
where $s_1 \approx 5.0538411 - 7.1194683 i$ is an algebraic number of degree
$8$. To be very explicit, $\frac{s_1}{52 \alpha}$ has minimal polynomial
\[ x^8 + 14197606 x^6 - 56569153 x^4 + 15962594 x^2 + 175561. \]
In the final step, we relate, in a standard way, the Eisenstein series $E_2
(\tau)$ at our specific point of interest to an algebraic
combination of $\frac{1}{\pi}$ and $\theta_3 (\tau)^4$. The fact
that this is possible whenever $\tau$ is a quadratic irrationality is what
lies behind the abundance of series for $\frac{1}{\pi}$.

\begin{proposition}
  With $\alpha = k^2 (\tau_0)$ and $\tau_0 = \frac{1}{2} +
  \frac{3}{10} \sqrt{- 5}$ as in Proposition \ref{prop:tau}, we have
  \begin{equation}
    E_2 (\tau_0) = \frac{2 \sqrt{5}}{\pi} + s_2 F \left( \alpha
    \right)^2 \label{eq:E2as1piF}
  \end{equation}
  where $s_2$ is an algebraic number explicited at the end of the proof.
\end{proposition}

\begin{proof}
  The main tool of this evaluation is the fact that Ramanujan's multiplier of
  the second kind, given by {\cite[Chapter 5]{borwein-piagm}}
  \begin{equation}
    R_p \left( l, k \right) \assign \frac{p E_2 (p\tau) - E_2
    (\tau)}{\theta_3^2 (p\tau) \theta_3^2 (\tau)}, \label{eq:mult2}
  \end{equation}
  is an algebraic function of $l \assign k (p\tau)$ and $k
  \assign k (\tau)$ (and hence an algebraic number when $\tau$ is
  a quadratic irrationality). For our specific value $\tau = \tau_0$ we will
  have use of the three instances
  \begin{align*}
    R_2 \left( l, k \right) & = l' + k,\\
    R_3 \left( l, k \right) & = 1 + k l + k' l',\\
    R_5 \left( l, k \right) & = \left( 3 + k l + k' l' \right) \sqrt{\frac{1
    + k l + k' l'}{2}} .
  \end{align*}
  By differentiating the transformation law for the Dedekind eta function, we
  also have
  \begin{equation}
    E_2 \left( - \frac{1}{\tau} \right) = \tau^2 E_2 (\tau) +
    \frac{6 \tau}{\pi i} . \label{eq:E2law}
  \end{equation}
  If $\tau = i \sqrt{p}$ then $- 1 / \tau = i / \sqrt{p} = \tau / p$. In that
  case (\ref{eq:mult2}) and (\ref{eq:E2law}) give two relations between $E_2
  (\tau)$ and $E_2 \left( \tau / p \right)$. Also note that, for
  such $\tau$, we have $l = k'$ and $l' = k$ by (\ref{eq:ktr}). In the case,
  $p = 5$ we thus obtain
  \begin{align}
    E_2 \left( i / \sqrt{5} \right) & = \frac{3 \sqrt{5}}{\pi} - \frac{1}{2
    \sqrt{5}} \theta_3 \left( i / \sqrt{5} \right)^4 \left( 3 + 2 k k' \right)
    \sqrt{\frac{1 + 2 k k'}{2}} \nonumber\\
    & = \frac{3 \sqrt{5}}{\pi} - \frac{\sqrt{1 + \sqrt{5}}}{\sqrt{10}}
    \theta_3 \left( i / \sqrt{5} \right)^4 .  \label{eq:E2i5}
  \end{align}
  where in the second step we used the singular values {\cite[Chapter
  4]{borwein-piagm}}
  \begin{align*}
    k \left( i \sqrt{5} \right) & = \frac{\sqrt{\sqrt{5} - 1} - \sqrt{3 -
    \sqrt{5}}}{2},\\
    k' \left( i \sqrt{5} \right) & = \frac{\sqrt{\sqrt{5} - 1} + \sqrt{3 -
    \sqrt{5}}}{2} .
  \end{align*}
  Using the modular equations of degree $3$ (of both first and second kind),
  we now find an algebraic $r_1$ such that
  \[ 3 E_2 \left( 3 i / \sqrt{5} \right) - E_2 \left( i / \sqrt{5} \right) =
     r_1 \theta_3 \left( 3 i / \sqrt{5} \right)^4 \]
  and hence, by (\ref{eq:E2i5}), an algebraic number $r_2$ such that
  \[ E_2 \left( 1 + 3 i / \sqrt{5} \right) = E_2 \left( 3 i / \sqrt{5} \right)
     = \frac{\sqrt{5}}{\pi} + r_2 \theta_3 \left( 3 i / \sqrt{5} \right)^4 .
  \]
  A further and likewise application of the modular equations of degree $2$
  yields
  \[ E_2 \left( \frac{1}{2} + \frac{3}{10} \sqrt{- 5} \right) = \frac{2
     \sqrt{5}}{\pi} + s_2 \theta_3 \left( \frac{1}{2} + \frac{3}{10} \sqrt{-
     5} \right)^4 \]
  where $s_2 \approx - 0.043464355 + 0.061229289 i$ has minimal polynomial
  \begin{align*}
    & 625 x^8 - 47597450000 x^7 + 64879599000 x^6 + 34024656000 x^5\\
    & - 58306698000 x^4 + 168524800 x^3 + 8089408640 x^2 + 722959360 x +
    44943616.
  \end{align*}
  In light of (\ref{eq:Kktheta}), this is the claimed (\ref{eq:E2as1piF}).
\end{proof}

\begin{theorem}
  \label{thm:sunconj}Sun's conjecture (\ref{eq:sunconj}) holds true.
\end{theorem}

\begin{proof}
  We combine the left-hand side of (\ref{eq:sunFE2}) with (\ref{eq:E2as1piF})
  to get
  \[ s_1 F (\alpha)^2 + 52 \sqrt{5} E_2 (\tau_0) =
     \left( s_1 + 52 \sqrt{5} s_2 \right) F (\alpha)^2 +
     \frac{520}{\pi} . \]
  The claim follows by verifying that $s_1 + 52 \sqrt{5} s_2 = 0$.
\end{proof}

\section{Further series of the same form}\label{sec:other}

The other conjectures in {\cite{sun-520}} of the same form as
(\ref{eq:sunconj}) can be proved analogously to (\ref{eq:sunconj}). We
therefore only provide the counterpart of Proposition \ref{prop:tau}.

\begin{proposition}
  \label{prop:taus}Let $x, y, \tau, p$ be as in the table below. Then
  \[ X = k' (\tau), \hspace{2em} Y = k' (p\tau), \]
  satisfy the two relations (\ref{eq:relXYxy}). Moreover, equation
  (\ref{eq:Ahyp}) holds.

  In the cases, where $p$ occurs starred, $Y=k'(p\tau+1)=1/k'(p\tau)$ has to be taken instead.
  
  \begin{table}[h]
    \begin{center}
    \begin{tabular}{|c|c|c|c|c|c|}
      \hline
      $x$ & $y$ & $\tau$ & $p$ & {\cite{sun-520}} & {Remark \ref{rk:Tbc}} \\
      \hline
      $\frac{i}{240}$ & $6 i$ & $i \sqrt{\frac{3}{10}}$ & $5$ & (3.11) & (IV5) \\
      \hline
      $\frac{i}{289}$ & $14 i$ & $i \sqrt{\frac{1}{7}}$ & $7$ & (3.12) & \cite[(33)]{wz-legendre} \\
      \hline
      $\frac{i}{2800}$ & $14 i$ & $i \sqrt{\frac{7}{10}}$ & $5$ & (3.13) & (IV11) \\
      \hline
      $\frac{i}{576}$ & $21 i$ & $i \sqrt{\frac{3}{14}}$ & $7$ & (3.14) & (IV9) \\
      \hline
      $\frac{i}{46800}$ & $36 i$ & $i \sqrt{\frac{13}{10}}$ & $5$ & (3.15) & (IV13) \\
      \hline
      $\frac{i}{2304}$ & $45 i$ & $i \sqrt{\frac{5}{14}}$ & $7$ & (3.16) & (IV10) \\
      \hline
      $\frac{i}{439280}$ & $76 i$ & $i \sqrt{\frac{19}{10}}$ & $5$ & (3.17) & (IV17) \\
      \hline
      $-\frac{i}{29584}$ & $175 i$ & $\frac{1}{2} + \frac{i}{2}\sqrt{\frac{19}{7}}$ & $7$ & (3.18) & (IV21) \\
      \hline
      $\frac{i}{5616}$ & $300 i$ & $i \sqrt{\frac{3}{26}}$ & $13$ & (3.19) & (IV14) \\
      \hline
      $\frac{i}{28880}$ & $1156 i$ & $i \sqrt{\frac{5}{26}}$ & $13$ & (3.20) & (IV16) \\
      \hline
      $\frac{i}{20400}$ & $1176 i$ & $i \sqrt{\frac{3}{34}}$ & $17$ & (3.21) & (IV15) \\
      \hline
      $\frac{i}{243360}$ & $12321 i$ & $i \sqrt{\frac{5}{38}}$ & $19$ & (3.22) & (IV18) \\
      \hline
      $\frac{1}{48}$ & $\frac{15}{16}$ & $\frac12 + \frac{i}{2} \sqrt{\frac{5}{3}}$ & $3^\star$ & (3.23) & (IV1) \\
      \hline
      $\frac{1}{480}$ & $8$ & $\frac{1}{2} + \frac{3 i}{2} \sqrt{\frac{1}{5}}$ & $5$ & (3.24) & (IV2) \\
      \hline
      $\frac{1}{5760}$ & $18$ & $\frac{1}{2} + \frac{i}{2} \sqrt{\frac{17}{5}}$ & $5$ & (3.25) & (IV3) \\
      \hline
      $-\frac{1}{48}$ & $\frac98$ & $i \sqrt{\frac{1}{3}}$ & $3^\star$ & (3.26) & (IV19) \\
      \hline
      $-\frac{1}{288}$ & $\frac{225}{16}$ & $i \sqrt{\frac{1}{7}}$ & $7^\star$ & (3.27) & (IV20) \\
      \hline
    \end{tabular}
    \end{center}
    \caption{\label{tbl:taus}The modular parametrization for the additional entries from \cite{sun-520}}
  \end{table}
\end{proposition}

\begin{theorem}
  Identities (3.11) to (3.27) in {\cite{sun-520}} hold.
\end{theorem}

\begin{proof}
  Using Proposition \ref{prop:taus} in place of Proposition \ref{prop:tau},
  the proofs are analogous to the proof of (\ref{eq:sunconj}) which was
  completed in Theorem \ref{thm:sunconj}. (The starred cases need some
  slight adjustments to Proposition \ref{prop:FGb} which continues to hold
  in light of Remark \ref{rk:FGb}.)
\end{proof}

\begin{remark}
  \label{rk:Tbc}Let $T_k (b,c)$ denote, as in {\cite{sun-520}},
  the coefficient of $x^k$ in the expansion of $\left( x^2 + b x + c
  \right)^k$. In terms of the Legendre polynomials, these quantities can be
  expressed as {\cite{sun-143n}}
  \[ T_k (b,c) = \left( b^2 - 4 c \right)^{k / 2} P_k \left(
     \frac{b}{\sqrt{b^2 - 4 c}} \right) . \]
  Using (\ref{eq:AxyP}) we therefore have
  \[ A (x,y) = \sum_{k = 0}^{\infty} \left( - x^2 \right)^k 
     \binom{2 k}{k}^2 T_{2 k} \left( \sqrt{4 + \frac{y}{x}}, 1 \right) . \]
  For instance,
  \[ A \left( \frac{1}{480}, 8 \right) = \sum_{k = 0}^{\infty} \frac{1}{\left(
     - 480^2 \right)^k}  \binom{2 k}{k}^2 T_{2 k} \left( 62, 1 \right) \]
  which we notice as a building block for entry (IV2) in {\cite{sun-520,sun-143n}},
  \begin{equation}
    \sum_{k = 0}^{\infty} \frac{340 k + 59}{\left( - 480^2 \right)^k} 
    \binom{2 k}{k}^2 T_{2 k} \left( 62, 1 \right) = \frac{120}{\pi},
    \label{eq:sunIV2}
  \end{equation}
  thus connecting entries (3.24) and (IV2). Similar relations exist for
  all of the cases considered in Proposition \ref{prop:taus} and for
  each case we have recorded the corresponding entry of {\cite{sun-520}} in
  Table \ref{tbl:taus} (in the second case the sum was discovered in \cite{wz-legendre}).
  We remark that the cases in group (IV) not thus connected by Table \ref{tbl:taus}
  correspond to series arising from \eqref{eq:Fdef} where $y=i$.
  
  We remark that the equivalence of (\ref{eq:sunIV2}) and (\ref{eq:sunconj})
  relies on
  \begin{equation}\label{eq:ADrel}
    2 A \left( \frac{1}{480}, 8 \right) - 28 \left( \theta_x A \right) \left(
    \frac{1}{480}, 8 \right) + 65 \left( \theta_y A \right) \left(
    \frac{1}{480}, 8 \right) = 0
  \end{equation}
  where $\theta_z = z \frac{\md}{\md z}$ is the Euler differential
  operator. The truth of (\ref{eq:ADrel}) follows from the truth of
  (\ref{eq:sunconj}) and (\ref{eq:sunIV2}), which was proved in
  {\cite{wz-legendre}}.
  In fact, as pointed out by James Wan, \eqref{eq:ADrel} can be obtained by
  combining \eqref{eq:sunIV2} and its ``companion series'', see \cite[Section
  7]{cwz-legendre} and \cite[Remark 2]{wz-legendre}.  Underlying the concept of
  a companion series is that when one has two parameters, such as $x,y$, one
  can differentiate with respect to either of them. For our purposes only the
  derivative with respect to $x$ was considered but, clearly, Proposition
  \ref{prop:Ax} works equally well for $\frac{\md A}{\md y}$.  Pursuing this,
  one obtains another series for $1/\pi$, a companion of the original series.
  

\end{remark}

\vfill

\paragraph{Acknowledgements}
We are very grateful to James Wan and Wadim Zudilin for valuable comments and
encouragement. Their prior work was crucial to us. We also thank Zhi-Wei Sun
for sharing an updated list of his conjectural series as well as helpful
suggestions.

\end{document}